\documentclass[12pt]{amsart}

\author{Paul Pollack}
\address{University of British Columbia\\ Department of Mathematics \\ 1984 Mathematics Road\\ Vancouver, British Columbia V6T 1Z2, Canada}
\address{Simon Fraser University\\Department of Mathematics\\ Burnaby, British Columbia V5A 1S6, Canada}
\email{pollack@math.ubc.ca}
\author{Lola Thompson}
\address{Dartmouth College\\Department of Mathematics\\6188 Kemeny Hall\\Hanover, New Hampshire 03755, United States}
\email{lola.thompson@dartmouth.edu}
 
\title{On the degrees of divisors of $T^n-1$}
\usepackage{amsmath,amssymb,amsthm,geometry,url}
\DeclareMathAlphabet{\curly}{U}{rsfs}{m}{n}
\newtheorem{thm}{Theorem}[section]
\newtheorem{prop}[thm]{Proposition}

\newtheorem{cor}[thm]{Corollary}

\newtheorem{lem}[thm]{Lemma}
\theoremstyle{definition}

\newtheorem*{rmk}{Remark}
\newtheorem*{remarks}{Remarks}

\begin{document}
\def\phi{\varphi}
\renewcommand{\labelenumi}{(\roman{enumi})}
\def\polhk#1{\setbox0=\hbox{#1}{\ooalign{\hidewidth
    \lower1.5ex\hbox{`}\hidewidth\crcr\unhbox0}}}
\newcommand{\del}{\ensuremath{\delta}}
\def\A{\curly{A}}
\def\B{\curly{B}}
\def\e{\mathrm{e}}
\def\D{\curly{D}}
\def\E{\curly{E}}
\def\F{\mathbf{F}}
\def\C{\mathbf{C}}
\def\I{\curly{I}}
\def\N{\mathbf{N}}
\def\D{\curly{D}}
\def\Q{\mathbf{Q}}
\def\O{\curly{O}}
\def\V{\curly{V}}
\def\W{\curly{W}}
\def\Z{\mathbf{Z}}
\def\p{\tilde{p}}
\def\Pp{\curly{P}}
\def\pr{\mathfrak{p}}
\def\Proj{\mathbf{P}}
\def\q{\mathfrak{q}}
\def\Ss{\curly{S}}
\def\T{\curly{T}}
\def\Nm{\mathcal{N}}
\def\cont{\mathrm{cont}}
\def\ord{\mathrm{ord}}
\def\rad{\mathrm{rad}}
\def\lcm{\mathop{\mathrm{lcm}}}
\numberwithin{equation}{section}
\begin{abstract} Fix a field $F$. In this paper, we study the sets $\D_F(n) \subset [0,n]$ defined by
\[ \D_F(n):= \{0 \leq m \leq n: T^n-1\text{ has a divisor of degree $m$ in } F[T]\}. \]
When $\D_F(n)$ consists of all integers $m$ with $0 \leq m \leq n$, so that $T^n-1$ has a divisor of every degree, we call $n$ an \emph{$F$-practical number}. The terminology here is suggested by an analogy with the \emph{practical numbers} of Srinivasan, which are numbers $n$ for which every integer $0 \leq m \leq \sigma(n)$ can be written as a sum of distinct divisors of $n$. Our first theorem states that, for any number field $F$ and any $x \geq 2$, 
\[ \#\{\text{$F$-practical $n\leq x$}\} \asymp_{F} \frac{x}{\log{x}}; \]
this extends work of the second author, who obtained this estimate when $F=\Q$.

Suppose now that $x \geq 3$, and let $m$ be a natural number in $[3,x]$. We ask: For how many $n \leq x$ does $m$ belong to $\D_F(n)$? We prove upper bounds in this problem for both $F=\Q$ and $F=\F_p$ (with $p$ prime), the latter conditional on the Generalized Riemann Hypothesis. In both cases, we find that the number of such $n \leq x$ is $\ll_{F} x/(\log{m})^{2/35}$, uniformly in $m$.
\end{abstract}
\maketitle

\section{Introduction}
Let $F$ be a field. In this paper, we study the sets of nonnegative integers which appear as the set of degrees of divisors of $T^n-1$ in $F[T]$, i.e., the sets
\[ \D_F(n):= \{0 \leq m \leq n: \text{$T^n-1$ has a divisor of degree $m$ over $F$}\}. \]
When this set consists of all integers $0 \leq m \leq n$, we call $n$ an \emph{$F$-practical number}. For example, $6$ is a $\Q$-practical number, as shown by the following list of divisors of $T^6-1$:
\begin{align*} 1, ~\; T-1, ~\; T^2+T+1, ~\; T^3-1, ~\; T^4 + T^3-T-1, ~\; T^5+T^4+T^3+T^2+T+1, ~\; T^6-1. 
\end{align*}
It is easy to see directly (for example, by applying Gauss's lemma) that if $T^n-1$ has a divisor of a given degree over $\Q$, then it has a divisor of the same degree over $\Z$. As a consequence, for any field $F$, each $\Q$-practical number is also an $F$-practical number.

The distribution of $\Q$-practical numbers has been investigated by the second author \cite{thompson11}. Recall that with $\Phi_d(T)$ denoting the $d$th cyclotomic polynomial, we have
\begin{equation}\label{eq:universal} T^n-1 = \prod_{d \mid n} \Phi_d(T). \end{equation}
Over $\Q$, each of the right-hand factors $\Phi_d(T)$ is irreducible of degree $\phi(d)$. It follows that a natural number $n$ is $\Q$-practical precisely when every integer $m \in [0,n]$ can be written as a sum of terms $\phi(d)$, where $d$ runs over a subset of the divisors of $n$. 

The term ``$F$-practical number'' is suggested by an analogy between the $\Q$-practical numbers and \emph{Srinivasan's practical numbers} \cite{srinivasan48}, which are numbers $n$ for which every $m \in [0,\sigma(n)]$ can be written as as a sum of distinct divisors of $n$. Such $n$ have been studied by several authors, including Erd\H{o}s \cite{erdos50}, Hausman \& Shapiro \cite{HS84}, Tenenbaum \cite{tenenbaum86, tenenbaum95P}, and Saias \cite{saias97}. In the last of these papers, Saias shows that for all $x\geq 2$, 
\begin{equation}\label{eq:saias} \#\{\text{practical $n \leq x$}\} \asymp \frac{x}{\log{x}}. \end{equation}
Exploiting the analogy between practical numbers and $\Q$-practical numbers, the second author (op. cit.) proved the $\Q$-practical analogue of Saias's estimates:
\[ \#\{\text{$\Q$-practical $n\leq x$}\} \asymp \frac{x}{\log{x}}.\] 
(In the above statements, the notation ``$f\asymp g$'' means that we have both $f\ll g$ and $g\ll f$.)

One of our goals in this paper is to gain some understanding of the $F$-practical numbers over more general fields $F$. We begin by observing  that each cyclotomic polynomial $\Phi_d(T)$ always splits into (not necessarily distinct) irreducible factors of the same degree over $F$. This is easy to see in the case when the characteristic of $F$, say $p$, does not divide $d$ (for example, in characteristic zero). In this case, the roots of $\Phi_d(T)$ are exactly the $\phi(d)$ primitive $d$th roots of unity from the algebraic closure of $F$. Each primitive $d$th root of unity generates the same extension of $F$, and thus all irreducible factors of $\Phi_d(T)$ have the same degree, as desired. The case when $p$ divides $d$ reduces to the previous one, since then $\Phi_d(T) = \Phi_{d_{(p)}}(T)^{\phi(d/d_{(p)})}$, where $d_{(p)}$ denotes the largest divisor of $d$ coprime to $p$.

From the last paragraph, it makes sense to define an arithmetic function $\phi_F$ by letting $\phi_F(d)$ denote the common degree of each irreducible factor of $\Phi_d(T)$ over $F$. (For example, $\phi_F = \phi$ when $F=\Q$.) Then each $\Phi_d(T)$ is a product of $\phi(d)/\phi_F(d)$ (not necessarily distinct) irreducible polynomials of degree $\phi_F(d)$. So from \eqref{eq:universal}, $m$ is the degree of a divisor of $T^n-1$ precisely when there is a collection $\Ss$ of divisors of $n$ for which $m$  can be written in the form
\begin{equation}\label{eq:representation} m=\sum_{d \in \Ss} a_d \phi_F(d), \quad\text{where each}\quad 0 \leq a_d \leq \frac{\phi(d)}{\phi_F(d)}. \end{equation}
In \S\ref{sec:numberfield}, we use this criterion and some easy algebraic number theory to extend Thompson's theorem on $\Q$-practical numbers to an arbitrary number field. Note that since each $\Q$-practical number is automatically $F$-practical, it is enough to prove the upper bound estimate.

\begin{thm}\label{thm:numberfield} Let $F$ be a number field. Then for $x \geq 2$, the number of $F$-practical numbers in $[1,x]$ is $\ll_{F} \frac{x}{\log{x}}$. 
\end{thm} 

In her thesis (\cite{lolathesis}; see also \cite{thompson12}, \cite{thompson11P}), Thompson  studies the $F$-practical numbers also in the case when $F = \F_p$ (with $p$ prime). To discuss this case further, we need some notation. Write $\ell_p(d)$ for the multiplicative order of $p$ modulo $d$, assuming that $\gcd(d,p)=1$. In general, put $\ell_p^{\ast}(d) = \ell_p(d_{(p)})$, where $d_{(p)}$ denotes the largest divisor of $d$ coprime to $p$. As shown in \cite{thompson11P}, we have $\phi_{\F_p} = \ell_p^{\ast}$. Our limited understanding of the distribution of the numbers $\ell_p^{\ast}(d)$ is a significant obstacle to the study of $\F_p$-practical numbers. To work around this, Thompson assumes the \emph{Generalized Riemann Hypothesis} (GRH). (Throughout this paper, GRH always means the Riemann Hypothesis for Dedekind zeta functions.) Under this assumption, she shows (ibid.) that for $x\geq 3$,
\[ \frac{x}{\log{x}} \ll \#\{\text{$\F_p$-practical $n \leq x$}\} \ll_{p} x \sqrt{\frac{\log\log{x}}{\log{x}}}. \] 
The numerical data (see, for instance, \cite[Tables 1.2--1.4]{lolathesis}) suggests that for each fixed $p$, the true count of $\F_p$-practical numbers is $\sim C_p x/\log{x}$, as $x\to\infty$, where $C_p$ is a positive constant depending on $p$.

Up to this point, we have been discussing integers $n$ for which $\D_F(n)$ is the entire interval $[0,n]$. A weaker notion also suggests itself: Take an integer $m \leq x$ and count how often, among those $n \leq x$, one has $m \in \D_F(n)$. In other words, instead of requiring $T^n-1$ to have divisors of every degree, we fix in advance a target degree $m$. Our next theorem gives an upper bound in the case when $F=\Q$. It is convenient to label once and for all the so-called \emph{Erd\H{o}s--Ford--Tenenbaum constant}
\begin{equation}\label{eq:deltadef} \delta:= 1 - \frac{1+\log\log{2}}{\log{2}}. \end{equation}
Numerically, $\delta \approx 0.0860713$.

\begin{thm}\label{thm:EL} Fix a value $\delta'$ with $0 < \delta' < \delta$, where $\delta$ is defined in \eqref{eq:deltadef}. Then if $3 \leq m \leq x$, the number of $n \leq x$ for which $T^n-1$ has a divisor of degree $m$ in $\Q[T]$ is $\ll x/(\log{m})^{\delta'}$.
\end{thm}

Theorem \ref{thm:EL} should be viewed as analogous to a theorem of Erd\H{o}s, who considered \cite[p. 130]{erdos70} how often a target natural number $m$ could be written as a sum of distinct divisors of $n$. Indeed, our proof uses many of the same ideas. However, Erd\H{o}s was content to work with fixed values of $m$, whereas we seek a result with complete uniformity in $m$.  

Our last result is a GRH-conditional version of Theorem \ref{thm:EL} with $F=\F_p$ rather than $F=\Q$.
\begin{thm}[assuming GRH]\label{thm:polypractical} Fix a prime $p$. Suppose that $3 \leq m \leq x$.
\begin{enumerate}
	\item If $3 \leq m \leq x^{1-1/\log\log{x}}$, then the 
number of $n \leq x$ for which $T^n-1$ has a divisor of degree $m$  in $\F_p[T]$ is \[ \ll_{p} x/(\log{m})^{1/13}. \] 
\item If $x^{1-1/\log\log{x}} < m \leq x$, then the count of such $n$ is
\[ \ll_{p} x/(\log{m})^{2/35}. \]
\end{enumerate}
\end{thm}
\noindent The exponents $1/13$ and $2/35$ appearing above are close to the best our methods will yield. It would be interesting to know how close they are to being best possible.

One might compare Theorem \ref{thm:polypractical} with the result of Car \cite{car84} that in a wide range of $m$ and $n$, few polynomials of degree $n$ over $\F_p$ (or a general finite field $\F_q$) have a divisor of degree $m$. One must be cautious about such comparisons, however. For example, a typical polynomial of degree $n$ over $\F_p$ has about $n^{\log{2}}$ divisors (compare with \cite[Theorem 3.3.7]{KZ01}). However, for each fixed $A>0$, the polynomial $T^n-1$ has more than $\exp((\log{n})^{A})$ divisors on a set of $n$ of asymptotic density $1$. In fact, the same lower bounds holds almost always for the number $\phi(n)/\ell_p^{\ast}(n)$ of irreducible factors of $\Phi_n(T)$ in $\F_p[T]$; this follows from the normal order result for the Carmichael $\lambda$-function appearing as \cite[Theorem 2]{EPS91}.

A word about the organization of the paper:  We prove Theorem \ref{thm:numberfield} in \S\ref{sec:numberfield}. Theorem \ref{thm:EL} is proved in \S\ref{sec:proofthmEL}, after recalling some helpful results from the anatomy of integers. In \S\ref{sec:prep}, we review the GRH-conditional results needed for the proof of Theorem \ref{thm:polypractical}, which we present in \S\ref{sec:proofpolypractical}. We conclude the paper in \S\ref{sec:variants} by discussing some natural variants of the $\Q$-practical numbers. For example, we show that $2^{2^5}-1$ is the largest integer $n$ for which $T^n-1$ has exactly one monic divisor of each degree $0 \leq m \leq n$ in $\Q[T]$.

\subsection*{Notation} We write $\omega(n):=\sum_{p \mid n}1$ for the number of distinct prime factors of $n$ and $\Omega(n):=\sum_{p^k \mid n} 1$ for the number of prime factors of $n$ counted with multiplicity; $\Omega(n; y) := \sum_{p^k \mid n,~p \leq y} 1$ denotes the number of prime divisors of $n$ not exceeding $y$, again counted with multiplicity. The number of divisors of $n$ is denoted $d(n)$; for the number of divisors not exceeding $y$, we write $d(n; y)$. We use $P^{-}(m)$ and $P^{+}(m)$ for the smallest and largest prime factors of $m$, respectively, with the conventions that $P^{-}(1)=\infty$ and $P^{+}(1)=1$. An integer $n$ for which $P^{+}(n)\leq y$ is called \emph{$y$-smooth} (or \emph{$y$-friable}); the number of $y$-smooth $n\leq x$ is denoted $\Psi(x,y)$.

We use $\lambda(n)$ to denote the Carmichael $\lambda$-function, defined as the exponent of the finite abelian group $(\Z/n\Z)^{\times}$. For each natural number $n$ coprime to $a$, we write $\ell_a(n)$ for the multiplicative order of $a\bmod{n}$. For $n$ not necessarily coprime to $a$, we let $n_{(a)}$ denote the largest divisor of $n$ coprime to $a$, and we define $\ell_a^{\ast}(n) = \ell_a(n_{(a)})$. We call $\ell_a^{\ast}(n)$ the \emph{generalized order} of $a\bmod{n}$. (Note that $\ell_a^{\ast}(n)$ always divides $\lambda(n)$.)  When the intended value of $a$ is clear, we omit the subscripts on $\ell$ and $\ell^{\ast}$.

\section{Proof of Theorem \ref{thm:numberfield}}\label{sec:numberfield}
The proof of Theorem \ref{thm:numberfield} proceeds through a series of lemmas. The first of these, due to Stewart \cite{stewart54} and Sierpi\'nski \cite{sierpinski55}, characterizes Srinivasan's practical numbers in terms of their prime factorization.

\begin{lem}\label{lem:SS} Let $n$ be a natural number, and write the prime factorization of $n$ in the form $n = \prod_{i=1}^{r}p_i^{e_i}$, where each $e_i > 0$ and $p_1 < p_2 < \dots < p_r$. Let $j$ be the first index for which the inequality
\begin{equation}\label{eq:ssineq} p_j \leq 1+\sigma\left(\prod_{1 \leq i < j} p_i^{e_i}\right) \end{equation}
fails, where we take $j=r+1$ if no such index exists. Set
\begin{equation}\label{eq:nprime} n':= \prod_{1 \leq i < j} p_i^{e_i}. \end{equation}
Then every natural number $1 \leq m \leq \sigma(n')$ can be written as a sum of distinct divisors of $n$, but $\sigma(n')+1$ cannot be written as a sum of distinct divisors of $n$. Consequently, $n$ is practical precisely when  \eqref{eq:ssineq} holds for all indices $1 \leq j \leq r$.
\end{lem}

In what follows, we refer to $n'$, as defined in \eqref{eq:nprime}, as the \emph{practical component} of $n$. It can be shown (cf. \cite[Proposition 4]{margenstern91}) that the practical component $n'$ is the largest practical divisor of $n$.

For the remainder of the proof, we fix a number field $F$, and we consider $F$ as a subfield of a fixed algebraic closure $\overline{\Q}$ of $\Q$. We use $\zeta_d$ for a primitive $d$th root of unity from $\overline{\Q}$. In the next several lemmas, we show that if $n$ is $F$-practical, then there is a small multiple of $n$ that is practical in the sense of Lemma \ref{lem:SS}. The desired upper bound then follows from Saias's upper bound \eqref{eq:saias} on the count of practical numbers.

\begin{lem}\label{lem:ramify} Let $d$ be a natural number coprime to the (absolute) discriminant of $F$. Then $\phi_F(d)=\phi(d)$. 
\end{lem}
\begin{proof} Since the discriminant of $\Q(\zeta_d)$ divides $d^{\phi(d)}$ (see \cite[p. 269]{ribenboim72}), the number fields $F$ and $\Q(\zeta_d)$ have relatively prime discriminants. Since $F(\zeta_d)$ is the compositum of $F$ and $\Q(\zeta_d)$, we have (see \cite[p. 218]{ribenboim72}) 
	\[ [F(\zeta_d):\Q] = [F:\Q]\cdot [\Q(\zeta_d):\Q] = [F:\Q] \phi(d). \]
It follows that $\phi(d) = \frac{[F(\zeta_d):\Q]}{[F:\Q]} = [F(\zeta_d):F] =\phi_F(d)$, as claimed.
\end{proof}

\begin{lem}\label{lem:primeproduct} Let $p$ be a prime number. The product of the primes less than $p$ is always at least $p-1$.
\end{lem}
\begin{proof} This is easy to verify directly for primes $p < 5$. Now suppose that the claim has been shown for all primes smaller than $p$, where $p \geq 5$, and let $p'$ be the prime directly preceding $p$. Note that $p < 2p'$, by Bertrand's postulate. By the induction hypothesis, the product of the primes smaller than $p$ is at least 
	\[ p'(p'-1) \geq 3(p'-1) = 3p' - 3 > \frac{3}{2}p-3 \geq p-1, \]
since $p \geq 5$.
\end{proof}

\begin{lem}\label{lem:smallprime} If $n$ is $F$-practical and $p$ is the first prime not dividing $n$, then $pn$ is also $F$-practical.
\end{lem}
\begin{proof} We need to show that $T^{pn}-1$ has a divisor of degree $m$ over $F$ for all $0 \leq m \leq pn$. Since $\frac{T^{pn}-1}{T^n-1}$ has degree $(p-1)n$, and $T^n-1$ has a divisor of each degree in $[0,n]$, we see that $T^{pn}-1$ has a divisor of every degree $m$ with $(p-1)n \leq m \leq pn$.  So we can assume that $0 \leq m < (p-1)n$. 
	
	Write $m = (p-1)q+r$, where $0 \leq q < n$ and $0 \leq r < p-1$. Since $n$ is divisible by all primes $< p$, we have from Lemma \ref{lem:primeproduct} that $n \geq p-1 > r$. We are assuming that $n$ is $F$-practical, and so there is a divisor $f(T) \in F[T]$ of $T^n-1$ of degree $r$. That is, there is an $f(T) \in F[T]$ of degree $r$ for which 
	\begin{equation}\label{eq:fdividescyc} f(T) \mid \prod_{d \mid n} \Phi_d(T). \end{equation}
Similarly, since $q < n$, there is a divisor of $T^n-1$ of degree $q$. Such a divisor implies the existence of a representation (as in \ref{eq:representation})
\begin{equation}\label{eq:qsumrep} q = \sum_{d \mid n} a_d \phi_F(d), \quad\text{where}\quad 0 \leq a_d \leq \frac{\phi(d)}{\phi_F(d)}. \end{equation}
Multiplying \eqref{eq:qsumrep} by $p-1$, we obtain a representation
\begin{align} \notag (p-1)q &= \sum_{d \mid n} \left(a_d \frac{p-1} {\phi_F(pd)/\phi_F(d)}\right) \phi_F(pd)
	\\ &= \sum_{d \mid n} b_d \phi_F(pd), \quad\text{with each}\quad b_d:= a_d \frac{p-1} {\phi_F(pd)/\phi_F(d)}.\label{eq:sumrep2} \end{align}
With $F_d := F(\zeta_d)$, we have (noting that $p \nmid d$, since $p \nmid n$)
\[ \frac{\phi_F(pd)}{\phi_F(d)} = [F(\zeta_{pd}): F(\zeta_d)] = [F_d(\zeta_p): F_d] =\phi_{F_d}(p) \mid p-1, \]
and so all the $b_d$ are integers. Moreover, for each $d$ dividing $n$,
\[ 0 \leq b_d \leq \frac{\phi(d)}{\phi_F(d)} \frac{p-1} {\phi_F(pd)/\phi_F(d)} = \frac{\phi(pd)}{\phi_F(pd)}. \]
We now deduce from \eqref{eq:sumrep2} that there is a $g(T) \in F[T]$ of degree $(p-1)q$ for which
\begin{equation}\label{eq:gdividescyc}g(T) \mid \prod_{d \mid n} \Phi_{pd}(T).\end{equation} Combining \eqref{eq:fdividescyc} and \eqref{eq:gdividescyc}, we see that over $F$,
\[ f(T) g(T) \mid \left(\prod_{d \mid n} \Phi_{d}(T) \Phi_{pd}(T)\right) = T^{pn}-1, \]
and $fg$ has degree $r + (p-1)q = m$. So $fg$ is our sought-after divisor.
\end{proof}

Repeatedly applying Lemma \ref{lem:smallprime}, we arrive at the following result.

\begin{lem}\label{lem:smallprimecor} If $n$ is $F$-practical, then $\lcm[n, \prod_{p \leq z}p]$ is $F$-practical for every real number $z$.
\end{lem}

\begin{lem}\label{lem:convert2practical} Set $M:= \prod_{p \leq 2|D|}p$, where $D$ is the discriminant of $F$. If $n$ is $F$-practical, then $\lcm[n,M]$ is practical (in the sense of Srinivasan).
\end{lem}
\begin{proof} Put $N:= \lcm[n,M]$. By Lemma \ref{lem:smallprimecor}, $N$ is $F$-practical. We will show that $N$ satisfies the Stewart--Sierpi\'nski practicality criterion given in Lemma \ref{lem:SS}. Assuming $N$ is not practical, let $N'$ be the practical component of $N$. Then $N' < N$, and by Lemma \ref{lem:SS}, with $p$ denoting the smallest prime dividing $N/N'$, we have
	\begin{equation}\label{eq:bigp} p > \sigma(N')+1. \end{equation}
	
We must also have that $p > 2|D|$. To see this, observe that by construction, $N$ is divisible by all primes not exceeding $2|D|$. So if $p \leq 2|D|$, then $N'$ is divisible by all primes $< p$, and so by Lemma \ref{lem:primeproduct},
\[ 1 + \sigma(N') \geq 1+ \prod_{\substack{q < p\\q \text{ prime}}} (q+1) \geq 1 + \prod_{\substack{q < p\\q \text{ prime}}}q \geq 1 + (p-1)=p, \]
contradicting \eqref{eq:bigp}. Hence, $p > 2|D|$. 

We claim that $T^N-1$ has no divisor of degree $N'+1$, contradicting that $N$ is $F$-practical. Suppose contrariwise that
\begin{equation}\label{eq:sumexpr} N'+1 = \sum_{d \mid N} a_d \phi_F(d), \quad\text{where}\quad 0 \leq a_d \leq \frac{\phi(d)}{\phi_F(d)}. \end{equation}
The contribution to the sum in \eqref{eq:sumexpr} from divisors $d$ of $N'$ is bounded by
$\sum_{d\mid N'} \phi(d) = N'$; hence, there must be a $d$ dividing $N$ but not $N'$ which contributes to the right-hand side of \eqref{eq:sumexpr}. Since all the summands on the right-hand side of \eqref{eq:sumexpr} are nonnegative, clearly
\begin{equation}\label{eq:smaller} \phi_F(d) \leq N'+1. \end{equation}
Since $d$ divides $N$ but not $N'$, we can choose a prime $r$ dividing $\gcd(d, N/N')$. Clearly,
\[ r \geq P^{-}(N/N')=p > \max\{2|D|,\sigma(N')+1\}. \]
Since $r\mid d$ and $r \nmid D$, Lemma \ref{lem:ramify} shows that
\[ \phi_F(d) = [F(\zeta_d):F] \geq [F(\zeta_r):F] = \phi_F(r) = \phi(r) = r-1 \geq \sigma(N') +1. \]
Since $2 \mid N$, the practical component $N'$ of $N$ satisfies $N' \geq 2$, and so $\sigma(N') \geq N'+1$. Thus,
$\phi_F(d) \geq N'+2 > N'+1$, contradicting \eqref{eq:smaller}.
\end{proof}

\begin{proof}[Proof of Theorem \ref{thm:numberfield}] Define $M$ as in Lemma \ref{lem:convert2practical}. If $n \leq x$ is $F$-practical, then $dn$ is practical for some $d$ dividing $M$, namely $d = M/(M,n)$. Since $dn \leq dx$, the upper-estimate of \eqref{eq:saias} shows that the number of $F$-practical $n \leq x$ corresponding to this $d$ is $\ll dx/\log{x}$. Summing over the $O_F(1)$ divisors $d$ of $M$ completes the proof.
\end{proof}

\section{Proof of Theorem \ref{thm:EL}}\label{sec:proofthmEL}
The next few lemmas collect certain structural results about integers needed for the proof of Theorem \ref{thm:EL}. The first is a classical result of Landau (see \cite[Theorem 328, p. 352]{HW08}) giving the minimal order of the Euler $\phi$-function. 

\begin{lem}\label{lem:eulerphi} We have $\liminf_{n\to\infty} \frac{\phi(n)}{n/\log\log{n}} = e^{-\gamma}$. 
\end{lem}

Recall that $d(n; y)$ denotes the number of divisors of $n$ not exceeding $y$. The next lemma is implicit in \cite{erdos70}.

\begin{lem}\label{lem:smalldivisors} Let $x, y \geq 2$, and let $K \geq 1$. The number of integers $n \leq x$ with $d(n; y) \geq K$ is $\ll \frac{1}{K}x\log{y}$.
\end{lem}
\begin{proof} This is immediate from the first-moment estimate
	\[ \sum_{n \leq x} d(n; y) = \sum_{d \leq y}\sum_{\substack{n \leq x \\ d\mid n}}1 \leq x \sum_{d \leq y}\frac{1}{d}\ll x\log{y}. \qedhere\]
\end{proof}

The next result (easily deduced from \cite[Theorems 08--09, pp. 5--6]{HT88}; see also \cite[Exercise 04, p. 12]{HT88}) is an upper bound on the number of integers $n$ with an abnormally large number of prime factors. 

\begin{lem}\label{lem:HT2} Let $x \geq 3$. Uniformly for $0 < \kappa \leq 1.9$, the number of $n\leq x$ with $\Omega(n) > \kappa\log\log{x}$ is 
\[ \ll x/(\log{x})^{Q(\kappa)}, \quad\text{where}\quad Q(\kappa) = \kappa\log{\kappa}-\kappa+1. \]
\end{lem}

\begin{rmk} It is straightforward to check that the Erd\H{o}s--Ford--Tenenbaum constant $\delta$ of \eqref{eq:deltadef} satisfies $\delta = Q(1/\log{2})$. This property of $\delta$ will be important in what follows.
\end{rmk}

Write $H(x,y,z)$ for the count of $n \in [1,x]$ possessing a divisor from the interval $(y,z]$. The proof of Theorem \ref{thm:EL} requires fairly precise estimates for $H$. Conveniently, Ford \cite{ford08} has determined the order of magnitude of $H(x,y,z)$ in the complete space of parameters. His full result is somewhat complicated to state, but the next two lemmas isolate the special cases that are of interest to us (extracted from \cite[Theorem 1(v), (vi)]{ford08}). For our purposes, earlier results of Tenenbaum would also suffice (see, e.g., \cite[Theorem 21, pp. 29--30]{HT88}).

\begin{lem}\label{lem:ford1} Let $x > 10^5$. Suppose $y\geq 100$ and that $2y \leq z \leq y^2 \leq x$. Write $z = y^{1+u}$, so that $u = \log(z/y)/\log(y)$. Then
	\[ H(x,y,z) \asymp x u^{\delta} \big(\log \frac{2}{u}\big)^{-3/2}, \]
where $\delta \approx 0.08607$ is the constant defined in \eqref{eq:deltadef}.
\end{lem}

\begin{lem}\label{lem:ford2} Let $x > 10^5$. Suppose that $\sqrt{x} < y < z \leq x$. Suppose also that $z \geq y+1$ and $x/y \geq 1 + x/z$. Then 
	\[ H(x,y,z) \asymp H(x, x/z, x/y). \]
\end{lem}

We also need some understanding of the distribution of smooth numbers. The following upper bound is contained in work of de Bruijn \cite{dB66}. Recall that $\Psi(x,y)$ denotes the number of $y$-smooth numbers $n \leq x$. 

\begin{lem}\label{lem:tensmooth} For $2 \leq y \leq x$, set $u:= \frac{\log{x}}{\log{y}}$. Whenever $y \geq (\log{x})^2$ and $u \to \infty$, we have
	\[ \Psi(x,y) \leq \exp(-(1+o(1))u\log{u}). \]
\end{lem}

\begin{proof}[Proof of Theorem \ref{thm:EL}] We may suppose that $m$ (and hence also $x$) is large, since the assertion of the theorem is trivial for bounded values of $m$. We take two cases.
\vskip 5pt
\noindent \textsc{Case 1}: For the first half of the proof, we will assume that
\begin{equation}\label{eq:smallerm} m \leq x \exp(-\log{x}/\log\log{x}). \end{equation}
	
Suppose that $T^n-1$ has a divisor of degree $m$ in $\Q[T]$. We can assume that $n$ satisfies the inequality
\begin{equation}\label{eq:smalldivs} d(n; 2m\log\log{m}) < (\log{m})^2. \end{equation}
Indeed, by Lemma \ref{lem:smalldivisors}, the number of $n \leq x$ not satisfying \eqref{eq:smalldivs} is $\ll x/\log{m}$, which is negligible for us.

Since $T^n-1$ has a divisor of degree $m$, we can choose (as in \eqref{eq:representation}) a subset $\Ss$ of the divisors of $n$ with
\begin{equation}\label{eq:mphirep} m = \sum_{d\in \Ss} \phi(d). \end{equation}
If $d \in \Ss$, then $\phi(d) \leq m$, and so Lemma \ref{lem:eulerphi} implies that $d \leq 2 m \log\log{m}$. (We use here that $m$ is large and that $e^{\gamma} < 2$.) Thus, $\#\Ss < (\log{m})^2$ by \eqref{eq:smalldivs}. But then some term on the right-hand side of \eqref{eq:mphirep} must exceed $m/(\log{m})^2$. In particular, there must be some $d \in \Ss$ with
\[ 2m\log\log{m} \geq d \geq \phi(d) > m/(\log{m})^2. \]
Hence, $n$ is counted by 
\[ \tilde{H}:= H(x, m/(\log{m})^2, 2m\log\log{m}).\] We consider three cases:
\begin{itemize}
\item If $2m\log\log{m} \leq \sqrt{x}$, we apply Lemma \ref{lem:ford1} with $y=m/(\log{m})^2$, $z=2m\log\log{m}$. In this case, $u \asymp \frac{\log\log{m}}{\log{m}}$, and we find that
\[ \tilde{H} \ll x \left(\frac{\log\log{m}}{\log{m}}\right)^{\delta} (\log\log{m})^{-3/2} \ll x/(\log{m})^{\delta}, \]
as desired. 
\item If $m/(\log{m})^2 > \sqrt{x}$, then by Lemma \ref{lem:ford2},
\[ \tilde{H} \asymp H\bigg(x, \frac{x}{2m\log\log{m}}, x \frac{(\log{m})^2}{m}\bigg). \]
Recall we are assuming that $m$ satisfies \eqref{eq:smallerm}.  Apply Lemma \ref{lem:ford1} with $y=\frac{x}{2m\log\log{m}}$ and $z= x\frac{(\log{m})^2}{m}$, so that (using \eqref{eq:smallerm})
\begin{align*} u &\asymp \frac{\log\log{m}}{\log(x/(2m\log\log{m}))}\\
	&\ll \frac{(\log\log{m})(\log\log{x})}{\log{x}} \ll \frac{(\log\log{m})^2}{\log{m}}.\end{align*} We obtain  that 
	\[ \tilde{H} \ll x \left(\frac{(\log\log{m})^2}{\log{m}}\right)^{\delta} (\log\log{m})^{-3/2} \ll x/(\log{m})^\delta. \] 
\item It remains to treat the case when $m/(\log{m})^2 \leq \sqrt{x} < 2m\log\log{m}$.  In this case, 
$\sqrt{x}/(\log\sqrt{x})^3 \leq m/(\log{m})^2$ and $2m\log\log{m} < \sqrt{x}(\log{x})^3$. Thus,
\begin{align*} \tilde{H} &= H\left(x,\frac{m}{(\log{m})^2},\sqrt{x}\right) + H(x,\sqrt{x},2m\log\log{m}) \\
&\leq H\left(x, \frac{\sqrt{x}}{(\log{\sqrt{x}})^3}, \sqrt{x}\right) + H(x,\sqrt{x}, \sqrt{x} (\log{x})^3).	
\end{align*}
Applying Lemmas \ref{lem:ford1} and \ref{lem:ford2} as above, we find that both terms on the right-hand side are $\ll x/(\log{x})^{\delta} \ll x/(\log{m})^{\delta}$.
\end{itemize}
This completes the proof of Theorem \ref{thm:EL} in the case when $m$ satisfies \eqref{eq:smallerm}. In fact, in this case we obtain the upper bound claimed in the theorem with $\delta'$ replaced by the larger number $\delta$.

\vskip 5pt
\noindent \textsc{Case 2}: Now suppose \eqref{eq:smallerm} fails, i.e., that 
\begin{equation}\label{eq:mconstraints} x \exp(-\log{x}/\log\log{x})< m \leq x. \end{equation}

Let $n\leq x$ be such that $T^n-1$ has a divisor of degree $m$. We may assume that $p=P^{+}(n)$ satisfies \begin{equation}\label{eq:largestp} P^{+}(n) > \exp(2\log{x}/\log\log{x}).\end{equation} Indeed, by Lemma \ref{lem:tensmooth} (with $u = \frac{1}{2}\log\log{x}$), the number of $n \leq x$ not satisfying  \eqref{eq:largestp} is, for large $x$,  at most
\[ \frac{x}{\exp(\frac{1}{3}\log\log{x}\log\log\log{x})} < \frac{x}{\log{x}} \leq \frac{x}{\log{m}}, \]
which is negligible.

We fix $\epsilon > 0$ (depending only on $\delta'$) so that all but $O(x/(\log{x})^{\delta'})$ natural numbers $n \leq x$ satisfy the inequality
\begin{equation}\label{eq:Omegalower} \Omega(n) \leq \left(\frac{1}{\log{2}}-\epsilon\right)\log\log{x}. \end{equation}
Since $\delta = Q(1/\log{2})$ and $\delta' < \delta$, the possibility of choosing such an $\epsilon$ follows from Lemma \ref{lem:HT2} and the continuity of the function $Q(\kappa)$ appearing in the lemma statement. In what follows, we assume that \eqref{eq:Omegalower} holds. 

Since $T^n-1$ has a divisor of degree $m$, we may take a representation of $m$ in the form \eqref{eq:mphirep}, where $\Ss$ is a set of divisors of $n$. For each $d \in \Ss$ divisible by $p$, the number $\phi(d)$ is divisible by $p-1$. So reducing \eqref{eq:mphirep} modulo $p-1$, we find that
\begin{equation}\label{eq:divstar} m \equiv \sum_{d \in \T} \phi(d) \pmod{p-1}, \quad\text{where}\quad \T := \{d \in \Ss: p \nmid d\}. \end{equation}
Notice that $\T$ consists of divisors of $r:=n/p$. Also, from \eqref{eq:largestp}, we have  \[ r \leq x/\exp(2\log{x}/\log\log{x}).\]  Moreover, recalling \eqref{eq:mconstraints}, 
\begin{align}\notag m-\sum_{d \in \T}\phi(d) &\geq m - \sum_{d \mid r}\phi(d) \geq m-r  \\
	&\geq x \exp(-\log{x}/\log\log{x}) - x\exp(-2\log{x}/\log\log{x}) > 0.\label{eq:mnonzero}
\end{align}

We now count the possibilities for $n$ by first fixing $r$ and then using the relation \eqref{eq:divstar} to count the number of possibilities for $p$ given $r$. Since $\T$ consists entirely of divisors of $r$, the number of possibilities for $\T$, given $r$, is at most
\[ 2^{d(r)} < 2^{d(n)} \leq 2^{2^{\Omega(n)}} < \exp((\log{x})^{1-\frac{1}{2}\epsilon}). \]
(We use \eqref{eq:Omegalower} in the last step.) Rewriting \eqref{eq:divstar} in the form \[ p-1 \mid \left(m- \sum_{d \in \T}\phi(d)\right), \]
we see that given $\T$, the number of possibilities for $p$ is bounded by
\[ \max_{h \leq x} d(h) < \exp(\log{x}/\log\log{x}).\]
(We use here the maximal order of the divisor function, as in \cite[Theorem 317, p. 345]{HW08}.) Since $p$ and $r$ determine $n=pr$, the number of possibilities for $n$ is 
\begin{multline*} < \frac{x}{\exp(2\log{x}/\log\log{x})} \cdot \exp((\log{x})^{1-\frac{1}{2}\epsilon}) \cdot \exp(\log{x}/\log\log{x}) \\ < \frac{x}{\exp(\frac{1}{2}\log{x}/\log\log{x})} < \frac{x}{\log{x}}, \end{multline*}
which is negligible. This completes the proof.
\end{proof}

\begin{remarks}\mbox{}
\begin{enumerate}
\item We emphasize that what makes the $\F_p[T]$-situation considered in Theorem \ref{thm:polypractical} susbtantially more difficult than the $\Q[T]$ situation considered above is that no estimate of the type asserted in Lemma \ref{lem:eulerphi} holds with $\phi$ replaced by $\ell^{\ast}$.
\item It would be desirable to have a sharp lower bound to complement the upper bound in Theorem \ref{thm:EL}. An easy adaptation of the methods of \cite{PT12} gives the following related estimate: \emph{If $3 \leq m \leq \frac{1}{2}x$, then the number of $n \in [1,x]$ for which $T^n-1$ has a divisor of each degree in $[0,m]$ is $\gg x/\log{m}$.} 
\end{enumerate}
\end{remarks}

\section{Proof of Theorem \ref{thm:polypractical}}\label{sec:practicalpoly}
\subsection{Preliminary estimates}\label{sec:prep} Throughout \S\ref{sec:practicalpoly}, we assume that $a > 1$ is a fixed integer, and we write $\ell^{*}(n)$ for the generalized order of $a$ modulo $n$. In \S\ref{sec:prep}, we collect some known results on the behavior of $\ell^{*}(n)$ and the closely associated function $\lambda(n)$. These estimates will eventually be applied to prove Lemma \ref{lem:tech}, which will be the key component of our demonstration of Theorem \ref{thm:polypractical}. 

\begin{rmk} For the rest of \S\ref{sec:practicalpoly}, we suppress any dependence of implied constants on $a$.\end{rmk}

The following lemma, due to Kurlberg and Pomerance \cite[Theorem 23]{KP05}, shows that under GRH the numbers $\ell(p)$ are usually close to $p-1$.

\begin{lem}[assuming GRH]\label{lem:KP} Uniformly for $1 \leq y \leq \log{x}$, the number of primes $p \leq x$ (not dividing $a$) for which $\ell(p) \leq p/y$ is 
	\[ \ll \frac{\pi(x)}{y} + \frac{x}{(\log{x})^2}\log\log{x}. \]
\end{lem}

\begin{lem}[assuming GRH]\label{lem:KPcons} Let $x\geq 3$. The number of primes $p \leq x$ coprime to $a$ with $\ell(p) \leq p/(\log{p})$ is $\ll \frac{x}{(\log{x})^2}\log\log{x}$.
\end{lem}
\begin{proof} We can restrict our attention to $p > \sqrt{x}$. Then $\ell(p) \leq p/\log{p}< 2p/\log{x}$, and the estimate follows from Lemma \ref{lem:KP} with $y=\frac{1}{2}\log{x}$.
\end{proof} 

The next lemma is a special case of a result of Gottschlich \cite[Lemma 2.3]{gottschlich12}. 

\begin{lem}\label{lem:avram} Let $\Pp$ be a set of primes. Suppose that for  certain constants $\theta_1 > 1$, $\theta_2 > 0$, the number of elements of $\Pp$ not exceeding $x$ is
	\[ \ll \frac{x}{(\log{x})^{\theta_1}} (\log\log{x})^{\theta_2}, \]
for all $x\geq 3$. Then for $x \geq 3$, the number of integers $n \leq x$ all of whose prime factors belong to $\Pp$ is also
	\[ \ll \frac{x}{(\log{x})^{\theta_1}} (\log\log{x})^{\theta_2}, \]
where the implied constant depends at most on $\Pp$ and the $\theta_i$.
\end{lem}

\begin{lem}[assuming GRH]\label{lem:avramcons} Let $x\geq 3$. The number of $n \leq x$ all of whose prime factors $p$ either
\begin{enumerate}
	\item divide $a$, \emph{or}
	\item satisfy $\ell(p) \leq p/\log{p}$
\end{enumerate} is $\ll \frac{x}{(\log{x})^2} \log\log{x}$.
\end{lem}
\begin{proof} We let $\Pp$ be the set of primes $p$ dividing $a$ or satisfying $\ell(p) \leq p/\log{p}$. Since there are only $O(1)$ primes dividing $a$, Lemma \ref{lem:KPcons} shows that the hypotheses of Lemma \ref{lem:avram} are satisfied with $\theta_1=2$ and $\theta_2=1$.
\end{proof}

Finally, we recall an estimate of Friedlander, Pomerance, and Shparlinski \cite{FPS01} for the number of occurrences of small values of the Carmichael $\lambda$-function. 

\begin{lem}\label{lem:FPS} Suppose that $x$ is sufficiently large and that $\Delta \geq (\log\log{x})^3$. Then the number of $n \leq x$ with $\lambda(n) \leq n \exp(-\Delta)$ is at most 
	\[ x \exp(-0.69 (\Delta \log \Delta)^{1/3}). \]
\end{lem}

\subsection{Key lemmas}\label{sec:keylemmas}

In this section, we present several lemmas that play an important role in the proof of Theorem \ref{thm:polypractical}. The following lemma is a close cousin of Lemma \ref{lem:smalldivisors}, but the proof is somewhat more intricate. It should also be compared with Lemma \ref{lem:HT2}, which gives a sharper result but only under more restrictive hypotheses.

\begin{lem}\label{lem:HT} Let $x, y \geq 2$, and let $k \geq 1$. The number of $n \leq x$ with $\Omega(n;y) \geq k$ is $\ll \frac{k}{2^k} x \log{y}$.
\end{lem}

\begin{rmk} Taking $y=x$, we see that the number of $n \leq x$ with $\Omega(n) \geq k$ is $\ll \frac{k}{2^k}x\log{x}$.
\end{rmk} 

\begin{proof} The proof is almost identical to that suggested in Exercise 05 of \cite[p. 12]{HT88}, but we include it for the sake of completeness. Let $v:=2-1/k$. Let $g$ be the arithmetic function determined through the convolution identity $v^{\Omega(n;y)}= \sum_{d \mid n}g(d)$. Then $g$ is multiplicative. For $e\geq 1$, we have $g(p^e) = v^e-v^{e-1}$ if $p \leq y$, and $g(p^e) = 0$ if $p > y$. Hence,
\begin{align*} \sum_{n \leq x}v^{\Omega(n; y)} &= \sum_{d \leq x}g(d)\left\lfloor \frac{x}{d}\right\rfloor \leq x \sum_{d\leq x}\frac{g(d)}{d} \\&\leq x\prod_{p \leq y}\left(1 + \frac{v-1}{p} + \frac{v^2-v}{p^2}+ \dots\right) = \frac{x}{2-v}\prod_{3 \leq p \leq y}\left(1 + \frac{v-1}{p-v}\right).
\end{align*}
Now $2-v=1/k$, and the rightmost product is
\[ \leq \exp\left(\sum_{3 \leq p \leq y}\frac{v-1}{p-v}\right) \leq \exp\left(\sum_{3 \leq p \leq y}\frac{1}{p-2}\right) \leq \exp\left(\sum_{p \leq y}\frac{1}{p} + O(1)\right) \ll \log{y}.\] Collecting our estimates, we have shown that
\[ \sum_{n \leq x} v^{\Omega(n;y)} \ll k x\log{y}. \]
But each term with $\Omega(n;y) \geq k$ makes a contribution to the left-hand side that is $\geq v^k \geq (2-1/k)^k = 2^k (1-\frac{1}{2k})^k \gg 2^k$. Thus, the number of such terms is $\ll \frac{k}{2^k} x \log{y}$.
\end{proof}

We have already noted that there is no direct analogue for $\ell^{\ast}$ of the minimal order result for $\phi$ expressed in Lemma \ref{lem:eulerphi}. The following result is a partial workaround.

\begin{lem}[assuming GRH]\label{lem:tech} Fix $\theta$ with $0 < \theta \leq \frac12$. Suppose that $3 \leq y \leq x$. The number of integers $n \leq x$ which have a divisor $d > y$ satisfying 
\[ \ell^{\ast}(d) \leq d/\exp(4(\log{d})^{\theta}) \]
is $\ll_{\theta} x \log\log{y}/(\log{y})^{\theta}$.
\end{lem}
\begin{proof} Throughout the argument, we suppress the dependence of implied contants on $\theta$. We may assume always that $y$ is large, since the lemma is trivial for bounded values of $y$. For real $t\geq 1$, define the three sets
\begin{align*} \E_1(t)&:=\{e \leq t: e \text{ squarefull}\}, \\
 \E_2(t)&:=\{e \leq t: p \mid e\Rightarrow (p\mid a \text{ or }\ell(p) \leq p/\log{p})\}, \\
 \E_3(t)&:=\{e\leq t: \lambda(e) \leq e/\exp((\log{e})^{\theta})\}. \end{align*}
We set $e_1$, $e_2$, and $e_3$ equal to the largest divisors of $n$ from the three sets $\E_1, \E_2, \E_3$, respectively. We start by showing that we can assume each of the following inequalities:
\begin{align}
	e_1 &\leq \log{y}, \label{eq:e1}\\
	e_2 &\leq \exp((\log{y})^{\theta}), \label{eq:e2} \\
	e_3 &\leq y.\label{eq:e3}
\end{align}

It is easy to dispense with \eqref{eq:e1}. Indeed, by partial summation and the well-known estimate $\#\E_1(t) \ll \sqrt{t}$, the number of $n \leq x$ with a squarefull divisor larger than $\log{y}$ is $\ll x/(\log{y})^{1/2}$. This is acceptable for us, since $\theta \leq \frac12$. To see that we can assume \eqref{eq:e2}, note that the number of exceptional values of $n\leq x$ is at most
\begin{align*} x \sum_{\substack{e > \exp((\log{y})^{\theta}) \\ e\in\E_2(x)}}\frac{1}{e} &\leq x\left(\frac{\#\E_2(x)}{x}+ \int_{\exp((\log{y})^{\theta})}^{x} \frac{\#\E_2(t)}{t^2}\,dt\right) \\
	&\ll \frac{x}{(\log{x})^2}\log\log{x} + \frac{x}{(\log{y})^{\theta}}\log\log{y} \ll \frac{x}{(\log{y})^{\theta}} \log\log{y},
\end{align*}
where we have used the estimate of Lemma \ref{lem:avramcons} for $\#\E_2$.

It remains to justify the assumption \eqref{eq:e3}. We first estimate the counting function $\#\E_3(t)$. If $e$ is counted by $\#\E_3(t)$, then either $e \leq\sqrt{t}$ or $\lambda(e) \leq e/\exp((\log{\sqrt{t}})^{\theta})$. Lemma \ref{lem:FPS}, with $x=t$ and $\Delta=(\log\sqrt{t})^{\theta}$, thus implies that for large $t$,
\[ \#\E_3(t) \ll \sqrt{t} + t \exp(-(\log{t})^{\theta/3}) \ll t/(\log{t})^2. \]
Consequently, the number of $n \leq x$ with a divisor $e>y$ belonging to $\E_3$ is 
\[ \ll x \sum_{\substack{e > y \\ e \in \E_3(x)}}\frac{1}{e} \leq x \left(\frac{\#\E_3(x)}{x} + \int_{y}^{x}\frac{\#\E_3(t)}{t^2}\,dt\right) \ll \frac{x}{\log{y}}, \]
which is negligible for us. 

In addition to the conditions \eqref{eq:e1}--\eqref{eq:e3}, we may also suppose that $n$ does not have any divisor $d > y$ with $\Omega(d) \geq 10 \log\log{d}$. To see this, suppose for the sake of contradiction that $d$ is such a divisor. In the case when $d> x^{1/2}$, this implies that
\[ \Omega(n) \geq \Omega(d) \geq 10\log\log{d} \geq 9\log\log{x}. \]
But the number of $n \leq x$ with $\Omega(n) \geq 9\log\log{x}$ is $\ll x/(\log{x})^5$ by Lemma \ref{lem:HT}, and this is negligible for us. If $d \leq \sqrt{x}$, we can choose an integer $j \geq 0$ with 
\[ y^{2^j} < d \leq y^{2^{j+1}} \leq x. \]
Then with $z = y^{2^{j+1}}$, we have
\[ \Omega(n; z) \geq \Omega(d) \geq 10\log\log{d} \geq 10\log\log(z^{1/2}) \geq 9\log\log{z}, \]
and by Lemma \ref{lem:HT} again, the number of such $n \leq x$ is
\[ \ll \frac{x}{(\log{z})^5} \ll 2^{-5j} \frac{x}{(\log{y})^5}.\]
Summing over $j$, we see that the number of possible values of $x$ that can arise this way is $\ll x/(\log{y})^5$, which is acceptable.

We will show that for all values of $n$ which remain, every divisor $d > y$ of $n$ satisfies
\begin{equation}\label{eq:desired} \ell^{\ast}(d) > d/\exp(4(\log{d})^{\theta}). \end{equation}
From the last paragraph, we have
\[ \Omega(d) < 10\log\log{d}. \]
Put $d=d_1 d_2 q$, where $d_1$ is the largest divisor of $n$ from $\E_1$ and $d_2$ is the largest divisor of $d/d_1$ from $\E_2$. Then $q$ is squarefree and relatively prime to $a$, and $\ell(p) > p/\log{p}$ for every prime $p$ dividing $q$. Moreover,
\begin{equation}\label{eq:d1bound} d_1 \leq e_1 \leq \log{y} \leq \log{d} \end{equation}
and 
\begin{equation}\label{eq:d2bound} d_2 \leq e_2 \leq \exp((\log{y})^{\theta}) \leq \exp((\log{d})^{\theta}). \end{equation}
Since $d > y$ but $e_3 \leq y$, it follows that  $d \not \in \E_3$, and so
\begin{equation}\label{eq:lambdalower} \lambda(d) > d/\exp((\log{d})^{\theta}). \end{equation}
Because $d=d_1 d_2 q$ with $d_1$, $d_2$, and $q$ supported on disjoint sets of primes,
\[ \lambda(d) = \mathrm{lcm}[\lambda(d_1), \lambda(d_2), \lambda(q)] \leq \lambda(q) d_1 d_2. \]
Hence, estimates \eqref{eq:d1bound}, \eqref{eq:d2bound}, and \eqref{eq:lambdalower} yield
\begin{equation}\label{eq:lambdaqlower} \lambda(q) \geq \frac{\lambda(d)}{d_1 d_2} \geq \frac{d}{\exp((\log{d})^{\theta})} (\log{d})^{-1} \exp(-(\log{d})^{\theta}) > \frac{d}{\exp(3(\log{d})^{\theta})}. \end{equation}
For each prime $p$ dividing $q$, write $p-1 = \ell(p)\iota(p)$, so that $\iota(p)$ is the index of the subgroup of $\F_p^{\times}$ generated by $a$; from the definition of $q$,
\[ \iota(p) < \frac{p}{\ell(p)} \leq \log{p} \leq \log{d} \]
for all $p$ dividing $q$. Also,
\[ \ell(q) = \lcm_{p \mid q}[\ell(p)] = \lcm_{p \mid q}\left[\frac{p-1}{\iota(p)}\right] \geq \frac{\lcm_{p \mid q}[p-1]}{\prod_{p \mid q}\iota(p)} = \frac{\lambda(q)}{\prod_{p \mid q} \iota(p)}.\]
Thus, from \eqref{eq:lambdaqlower}, the bound $\iota(p) \leq \log{d}$, and the inequality $\omega(q)\leq \Omega(d) < 10\log\log{d}$,
\begin{align*} \ell(q) &\geq \frac{d}{\exp(3(\log{d})^{\theta})} \bigg(\prod_{p \mid q} \iota(p)\bigg)^{-1} 
\\&\geq \frac{d}{\exp(3(\log{d})^{\theta})} (\log{d})^{-10\log\log{d}} > \frac{d}{\exp(4(\log{d})^{\theta})}. \end{align*}
Since $q$ is a divisor $d$ that is coprime to $a$, we have that $\ell^{\ast}(d) \geq \ell(q)$, and so
\eqref{eq:desired} holds. This completes the proof of the lemma.
\end{proof}

We also need a simple observation concerning the behavior of the function $\phi/\ell^{\ast}$ along the divisor lattice (compare with \cite[Lemma 2]{FPS01}).

\begin{lem}\label{lem:monotone} If $d$ and $e$ are natural numbers for which $d \mid e$, then $\frac{\phi(d)}{\ell^{\ast}(d)} \mid \frac{\phi(e)}{\ell^{\ast}(e)}$.
\end{lem}
\begin{proof} By iteration, it suffices to treat the case when $e=qd$, where $q$ is a prime. We will prove the equivalent result that, in this case,
\begin{equation}\label{eq:divrelation} \frac{\ell^{\ast}(qd)}{\ell^{\ast}(d)} \mid \frac{\phi(qd)}{\phi(d)}. \end{equation}
We can assume that $q \nmid a$, since otherwise the left-hand ratio is $1$ and \eqref{eq:divrelation} holds trivially. We consider two cases, depending on whether or not $q$ divides $d$. If $q \nmid d$, then 
\[ \ell^{\ast}(qd) = \lcm[\ell(q), \ell^{\ast}(d)] \mid \lcm[q-1, \ell^{\ast}(d)]\mid (q-1)\ell^{\ast}(d). \]
Hence, 
\[ \frac{\ell^{\ast}(qd)}{\ell^{\ast}(d)}\mid q-1 = \frac{\phi(qd)}{\phi(d)}, \] i.e., \eqref{eq:divrelation} holds. Now suppose that $q \mid d$. Write $d = q^k d'$, where $q \nmid d'$. Then $\ell^{\ast}(qd) = \lcm[\ell(q^{k+1}), \ell^{\ast}(d')]$. Since $a^{\ell(q^k)}\equiv 1 \pmod{q^k}$, we have $a^{q \ell(q^k)}\equiv 1\pmod{q^{k+1}}$, and so $\ell(q^{k+1}) \mid q \ell(q^k)$. Thus,
\[ \ell^{\ast}(qd) = \lcm[\ell(q^{k+1}), \ell^{\ast}(d')] \mid \lcm[q \ell(q^k), \ell^{\ast}(d')] \mid q \lcm[\ell(q^k), \ell^{\ast}(d')] = q \ell^{\ast}(d), \]
which gives \eqref{eq:divrelation} in this case, noting that $\phi(qd)/\phi(d)=q$.
\end{proof}

\subsection{Completion of the proof of Theorem \ref{thm:polypractical}}\label{sec:proofpolypractical}

Throughout this section, we take $a=p$, where $\F_p$ is the field for which we are proving Theorem \ref{thm:polypractical}. Thus, $\ell^{\ast}(d)$ denotes the generalized order of $p$ modulo $d$. We continue to suppress the dependence of implied constants on $a$.

\begin{proof}[Proof of Theorem \ref{thm:polypractical}] We can always assume that $m$ is larger than any convenient constant (depending on $p$), since the theorem is trivial for bounded values of $m$. 
\vskip 5pt
\noindent \textsc{Case 1}: We suppose that
\begin{equation}\label{eq:msmallrange} 3\leq m \leq x \exp(-\log{x}/\log\log{x}). \end{equation}

Suppose that $T^n-1$ has a divisor of degree $m$ in $\F_p[T]$. By Lemma \ref{lem:tech}, with $y = m$ and $\theta = 0.079$, we may assume that every divisor $d$ of $n$ with $d > m$ satisfies 
\[ \ell^*(d) > d/\exp(4(\log d)^{0.079}); \] indeed, the number of exceptional $n$ is $O(x \log \log m/(\log m)^{0.079})$, which is small relative to our target upper bound. Since $m$ appears as the degree of a divisor of $T^n-1$, we can write \begin{equation}\label{eq:pprac} m = \sum_{d \mid n} \ell^*(d) a_d,\end{equation} where each $a_d$ satisfies $0 \leq a_d \leq \frac{\varphi(d)}{\ell^*(d)}.$ For each $d$ with $a_d > 0$, we have $\ell^*(d) \leq m.$ So, either $d \leq m$ or, by \eqref{eq:pprac}, we have 
\begin{equation}\label{eq:tocontradict}
d/\exp(4(\log d)^{0.079}) < \ell^*(d) \leq m.\end{equation}
The inequalities \eqref{eq:tocontradict} force $d < M$, where \[ M := m \exp(5(\log m)^{0.079}).\] Indeed, if we were to have $d > M$, then  \begin{align*} m \geq d/\exp(4(\log d)^{0.079}) &> M/\exp(4(\log M)^{0.079}) \geq \frac{m \exp(5(\log m)^{0.079})}{\exp(5(\log m)^{0.079})} = m,\end{align*} contradicting \eqref{eq:tocontradict}. Of course, if $d \leq m$, then it is also the case that $d \leq M$. So $d \leq M$ in any case.

Lemma \ref{lem:smalldivisors} allows us to assume that $d(n; M) < (\log m)^2$, since the exceptional set has size $O(x/\log{m})$. 
Referring back to \eqref{eq:pprac}, we see that there is a divisor $d$ of $n$ with $\ell^*(d)a_d \geq m/(\log m)^2.$ But $\ell^*(d) a_d \leq \varphi(d) < d,$ so $d > m/(\log m)^2.$ Therefore, $n$ has a divisor in the interval $(m/(\log m)^2, M]$ and so is counted by $H(x, m/(\log{m})^2, M)$.
We estimate the number of such $n \leq x$ using Lemmas \ref{lem:ford1} and \ref{lem:ford2}. 
As in the proof of Theorem \ref{thm:EL}, there are three cases to consider:

\begin{itemize}
\item If $M \leq \sqrt{x}$, we apply Lemma \ref{lem:ford1} directly, with $y = m/(\log m)^2$ and $z = M$. Then $\log(z/y) \asymp (\log m)^{0.079}$. On the other hand, $\log y \asymp \log m$. Thus, $u = \frac{\log(z/y)}{\log y} \asymp (\log m)^{-0.921}$. By Lemma \ref{lem:ford1}, \[ H(x, y, z) \asymp x u^\delta \left(\log \frac{2}{u}\right)^{-3/2} \ll \frac{x}{(\log m)^{0.921\delta}}\ll \frac{x}{(\log m)^{0.079}}.\]

\item If $\sqrt{x} < m/(\log m)^2$,  Lemma \ref{lem:ford2} gives $H\left(x, \frac{m}{(\log m)^2}, M\right) \asymp H\left(x, \frac{x}{M}, \frac{x}{m/(\log m)^2}\right)$. Now set  $y=x/M$ and $z=\frac{x}{m/(\log{m})^2}$. 
We are assuming  that $m \leq x^{1 - \frac{1}{\log \log x}}$, and so 
\[ \log{y}=  \log\frac{x}{M} = \log\left(\frac{x/m}{e^{5(\log m)^{0.079}}}\right) \gg \frac{\log x}{\log \log x}. \] Since $z/y = M(\log{m})^2/m < \exp(6(\log{m})^{0.079})$, we see that 
\[ u = \frac{\log{(z/y)}}{\log{y}} \ll \frac{(\log{m})^{0.079}}{\log x/\log\log{x}} \ll \frac{\log \log x}{(\log x)^{0.921}}. \]
So by Lemma \ref{lem:ford1}, \begin{align*} H\left(x, y,z \right) \ll x\left(\frac{\log \log x}{(\log x)^{0.921}}\right)^\delta (\log \log x)^{-3/2} \ll \frac{x}{(\log m)^{0.079}}.\end{align*} 

\item If $\frac{m}{(\log m)^2} \leq \sqrt{x} < M$, then we certainly have $\frac{\sqrt{x}}{\exp({6(\log \sqrt{x})^{0.079}})} \leq \frac{m}{(\log m)^2}$ and $M \leq \sqrt{x}\exp({6(\log x)^{0.079}})$. Thus, \begin{align} H\left(x, \frac{m}{(\log m)^2}, M\right) & = H\left(x, \frac{m}{(\log m)^2}, \sqrt{x}\right) + H\left(x, \sqrt{x}, M\right)\notag \\ &\label{eq:smallmfinal} \leq H\left(x, \frac{\sqrt{x}}{\exp(6 (\log \sqrt{x})^{0.079})}, \sqrt{x}\right) + H\left(x, \sqrt{x}, \sqrt{x}\exp(6(\log x)^{0.079})\right).\end{align} We may now apply Lemmas \ref{lem:ford1} and \ref{lem:ford2} as in the previous two cases to show that each term on the right-hand side of \eqref{eq:smallmfinal} is $O(x/(\log x)^{0.079}).$ 
\end{itemize}
This completes the proof of the Theorem \ref{thm:polypractical} in the case when $m$ satisfies \eqref{eq:msmallrange}.
\vskip 5pt
\noindent \textsc{Case 2}: We now suppose instead that \[ x \exp(-\log{x}/\log\log{x}) < m \leq x. \] 

Let $q= P^{+}(n)$. We will assume that $q > \exp(2\log{x}/\log\log{x})$; by Lemma \ref{lem:tensmooth}, this introduces an exceptional set of size $\ll x/\log{x}$, which is acceptable for us. Since $m$ appears as the degree of a divisor of $T^n-1$, we may write
\begin{equation}\label{eq:mrep} m = \sum_{d \mid n} a_d \ell^{\ast}(d), \quad\text{where each}\quad 0 \leq a_d \leq \frac{\phi(d)}{\ell^{\ast}(d)}. \end{equation}
Now consider \eqref{eq:mrep} modulo $\ell^{\ast}(q)$. Whenever $q \mid d$, we have $\ell^{\ast}(q) \mid \ell^{\ast}(d)$. So mod $\ell^{\ast}(q)$, the only divisors which contribute to the sum in \eqref{eq:mrep} are those $d$ not divisible by $q$, and all of these $d$ divide $r:=n/q$. Consequently,
\begin{equation}\label{eq:ourdiv} \ell^{\ast}(q) \mid \left(m - \sum_{d \mid r} a_d \ell^{\ast}(d)\right). \end{equation}
The right-hand side of relation \eqref{eq:ourdiv} is (cf. \eqref{eq:mnonzero}) at least
\[ m - \sum_{d \mid r} \phi(d) = m-r \geq x \exp(-\log{x}/\log\log{x}) - x\exp(-2\log{x}/\log\log{x}) > 0. \]
As in the proof of Case 2 of Theorem \ref{thm:EL}, our strategy will be to count, for each fixed $r$, the number of possibilities for $q$ allowed by \eqref{eq:ourdiv}. Since $q$ and $r$ determine $n=qr$, this will lead to an upper bound on the number of possible values of $n$.

To carry this plan out, it is convenient to impose some restrictions on $n$ additional to the lower bound on $q=P^{+}(n)$ assumed above, namely:
\begin{enumerate}
	\item $n > x/\log{x}$,
	\item $n$ satisfies the conditions of Lemma \ref{lem:tech} with
\[ \theta:=0.0579\quad\text{and}\quad y := \exp(\log{x}/\log\log{x}). \]
	\item $\Omega(n) \leq 1.359 \log\log{x}$.
\end{enumerate}
Clearly, (i) can be assumed excluding $O(x/\log{x})$ values of $n$, which is acceptable. By Lemma \ref{lem:tech}, the number of $n \leq x$ which are exceptions to (ii) is $\ll x/(\log{x})^{0.0578}$. Finally, by Lemma \ref{lem:HT2}, the number of $n \leq x$ which violate (iii) is 
$\ll x/(\log x)^{Q(1.359)} \ll x/(\log{x})^{0.0578}$. (Note that exponent $\frac{2}{35}$ claimed in this case of the theorem satisfies $\frac{2}{35} = 0.0571\ldots < 0.0578$.)

Since $r=n/q$ while $q > \exp(2\log{x}/\log\log{x})$, the number of possible $r$ is at most
\[ x/\exp(2\log{x}/\log\log{x}). \]
Given $r$, the inequalities governing the $a_d$ in \eqref{eq:mrep} imply that the number of possibilities for the right-hand side of \eqref{eq:ourdiv} is bounded by \begin{equation}\label{eq:product0} \prod_{d \mid r}\left(1 + \phi(d)/\ell^{\ast}(d)\right). \end{equation} By condition (ii) above, we have (using $n >x/\log{x} > y$)
\[ \ell^{\ast}(n) > n/\exp(4(\log{n})^{\theta}). \]
So by Lemma \ref{lem:monotone}, the product \eqref{eq:product0} is bounded above by
\[ \left(1 + \frac{\phi(n)}{\ell^{\ast}(n)}\right)^{d(n)} \leq \left(1 + \exp(4(\log{n})^{\theta})\right)^{d(n)} \leq \exp(O((\log{x})^{\theta} 2^{\Omega(n)})). \]
By condition (iii), $2^{\Omega(n)} \leq (\log{x})^{1.359 \log{2}}$, while $1.359 \log{2} + \theta < 0.9999$. So given $r$, the right-hand side of \eqref{eq:ourdiv} is determined in at most 
\[ \exp((\log{x})^{0.9999}) \]
ways, for large $x$. Since the right-hand side of \eqref{eq:ourdiv} is an integer in $[1,x]$, once it is fixed, the number of possibilities for its divisor $\ell^{\ast}(q)$ is at most
\[ \exp(\log{x}/\log\log{x}). \]
(We are using again the maximal order of the divisor function.)
Once more invoking condition (ii), we have (since $q > y^2 > y$)
\[ \frac{q-1}{\ell^{\ast}(q)} < \frac{q}{\ell^{\ast}(q)} < \exp(4(\log{q})^{\theta}) <  \exp(4(\log{x})^{\theta}); \]
since the ratio $(q-1)/\ell^{\ast}(q)$ is integral, we see that given $\ell^{\ast}(q)$, there are at most $\exp(4(\log{x})^{\theta})$ possibilities for $q$. 

Piecing everything together (determining successively $r$, the right-hand side of \eqref{eq:ourdiv}, $\ell^{\ast}(q)$, and finally $q$), the number of possibilities for $n=rq$ is bounded above by
\begin{multline*} \frac{x}{\exp(2\log{x}/\log\log{x})} \cdot \exp((\log{x})^{0.9999})\cdot \exp(\log{x}/\log\log{x}) \cdot \exp(4(\log{x})^{\theta})
\\ < \frac{x}{\exp(\frac{1}{2}\log{x}/\log\log{x})} < \frac{x}{\log{x}},
\end{multline*}
which is negligible. This completes the proof of the second case of the theorem, with the exponent $\frac{2}{35} = 0.0571\dots$ replaced by the larger number $0.0578$.
\end{proof}

\section{Concluding remarks: Variations on the $\Q$-practical numbers}
\label{sec:variants}
Srinivasan's practical numbers have a natural dual, namely, those $n$ for which each $m\in [0,\sigma(n)]$ has \emph{at most one} representation as a sum of distinct divisors of $n$. Call these \emph{efficient numbers}. Using the theory of sets of multiples, Erd\H{o}s showed \cite[Theorem 2]{erdos70} that the set of efficient numbers possesses a positive asymptotic density.

On the polynomial side, we define $n$ to be \emph{$\Q$-efficient} if $T^n-1$ has at most one monic divisor in $\Q[T]$ of each degree $m \in [0,n]$. Erd\H{o}s's argument, based on the theory of sets of multiples, may be adapted to show that the $\Q$-efficient numbers also have positive density. Indeed, this is immediate from the methods of \cite{erdos70} and the following lemma.

\begin{lem} If $\Ss$ is the set of natural numbers $n$ satisfying
\begin{enumerate}
	\item $n$ is not $\Q$-efficient,
	\item if $d \mid n$ and $d < n$, then $d$ is $\Q$-efficient,
	\item $\Omega(n) < 1.1 \log\log{(3n)}$,
\end{enumerate}
then the sum of the reciprocals of the members of $\Ss$ converges.
\end{lem}
\begin{proof} The proof is similar to Erd\H{o}s's argument and to our own proof of Case 2 of Theorem \ref{thm:EL}, so we provide only a sketch. By partial summation, it suffices to show that the counting function of $\Ss$ is $\ll x/(\log{x})^2$ for large $x$. Suppose that $n \in \Ss \cap [1,x]$. Since $n$ is not $\Q$-efficient, there are two monic divisors of $T^n-1$ of the same degree, and hence there is a nontrivial solution to the equation
\begin{equation}\label{eq:initrel} \sum_{d \mid n}\epsilon_d \phi(d) = 0, \quad\text{where each}\quad \epsilon_d \in \{-1, 0, 1\}. \end{equation}
(Here \emph{nontrivial} means that not all $\epsilon_d = 0$.) Put $p:=P^{+}(n)$. We can assume that
\[ P^{+}(n) > z^2, \quad\text{where}\quad z:=\exp(\log{x}/\log\log{x}), \]
since the number of exceptional $n\leq x$ is $\ll x/(\log{x})^2$ by Lemma \ref{lem:tensmooth}. We can also assume that $p:=P^{+}(n)$ divides $n$ only to the first power. Otherwise, $n$ has squarefull part $\geq p^2 > z^4$, and the number of such $n \leq x$ is $\ll x/z^2$, which is negligible. 

Consider \eqref{eq:initrel} modulo $p-1$. Whenever $p \mid d$, one has that $p-1 \mid \phi(d)$. So putting $r:=n/p$, it follows that
\begin{equation}\label{eq:divrelf} p-1 \mid \sum_{d \mid r} \epsilon_d \phi(d). \end{equation}
We claim that the right-hand side of \eqref{eq:divrelf} is a nonzero integer. If some $\epsilon_d$ appearing in \eqref{eq:divrelf} is nonzero, this is clear: In that case, the vanishing of the right-hand side of \eqref{eq:divrelf} implies that $T^r-1$ has two monic divisors of the same degree, contradicting condition (ii) in the definition of $\Ss$. But if all of the $\epsilon_d$ in \eqref{eq:divrelf} vanish, then the original sequence of $\epsilon_d$  appearing in \eqref{eq:initrel} is supported on multiples of $p$. In that case, after dividing \eqref{eq:initrel} through by $p-1$, we again obtain a contradiction to the $\Q$-efficiency of $r=n/p$.

Now we fix $r$ and count the number of $p$ allowed by \eqref{eq:divrelf}. Mimicking the end of the proof of Case 2 of Theorem \ref{thm:EL}, we find that the number of possible $n$ is
\[ \leq \frac{x}{\exp(2\log{x}/\log\log{x})} \cdot 2^{2^{1.1\log\log{(3x)}}}\cdot \exp(\log{x}/\log\log{x}) < x/\exp(\frac{1}{2}\log{x}/\log\log{x}), \]
once $x$ is large.
(We use here that $1.1 < 1/\log{2}$.) This last quantity is certainly $\ll x/(\log{x})^2$. 
\end{proof}

One might ask for both efficiency and practicality simultaneously, i.e., for numbers $n$ where each $m \in [0,\sigma(n)]$ has \emph{precisely one} representation as a sum of distinct divisors of $n$. The powers of $2$ have this property, and it is not so hard to show that these are all such $n$. The answer to the analogous polynomial problem is perhaps more unexpected. Define a \emph{$\Q$-optimal number} as an $n$ for which $T^n-1$ has precisely one monic divisor of each degree $m \in [0,n]$. In the remainder of this subsection, we classify the $\Q$-optimal numbers. 

The following lemma is due to the second author \cite[Lemma 4.1]{thompson11}.

\begin{lem}\label{lem:ssanalogue} Suppose that $n$ is $\Q$-practical. If $p$ is a prime not dividing $n$, then $pn$ is $\Q$-practical if and only if $p \leq n+2$. Moreover, $p^k M$ is $\Q$-practical, where $k \geq 2$, if and only if $p \leq n+1$. 
\end{lem}

Let $F_m := 2^{2^m}+1$ represent the $m$th \emph{Fermat number}. Below, we use the well-known result that if $p$ is an odd prime for which $p-1$ is a power of $2$, then $p = F_m$ for some $m$ (see \cite[p. 18]{HW08}); such a prime $p$ is called a \emph{Fermat prime}.

\begin{prop}\label{prop:qefficient} Let $k$ be a nonnegative integer.  Suppose that all of $F_0, F_1, \dots, F_{k-1}$ are prime. Then 
\begin{equation}\label{eq:optimaln} n := F_0 F_1 \cdots F_{k-1} \end{equation}
is a $\Q$-optimal number with $k$ distinct prime factors. Conversely, if there is any $\Q$-optimal number with $k$ distinct prime factors, then $F_0, \dots, F_{k-1}$ are all prime, and $n$ has the form \eqref{eq:optimaln}.
\end{prop}

\begin{proof}[Proof (sufficiency)]  Suppose that all of $F_0, \dots, F_{k-1}$ are prime, and define $n$ by \eqref{eq:optimaln}. The $\Q$-practicality of $n$ follows immediately from Lemma \ref{lem:ssanalogue} and the identity 
\begin{align*} F_0 F_1 \cdots F_{j-1}+2 &= (2^{2^0}-1)\left((2^{2^0}+1) (2^{2^1}+1) \cdots (2^{2^{j-1}}+1)\right) + 2 \\
	&= (2^{2^{j}} - 1) + 2 = F_{j}, \end{align*}
valid for all $j \geq 0$ (provided one interprets the empty product as $1$). Moreover, the identity \eqref{eq:universal} and the irreducibility of the cyclotomic polynomials implies that the number of monic divisors of $T^n-1$ is $2^{d(n)} = 2^{2^k}$, while the number of integers in $[0,n]$ is precisely
\[ n+1 = F_0 \cdots F_{k-1} + 1 = F_{k} - 1 = 2^{2^k}. \]
As these two numbers agree, the $\Q$-practicality of $n$ implies the $\Q$-efficiency of $n$ (by the pigeonhole principle). Hence, $n$ is $\Q$-optimal.\end{proof}
	
\begin{proof}[Proof (necessity)] There is nothing to prove if $k = 0$, so assume that $k \geq 1$. The key observation is that for each $\Q$-optimal number $n$, we have the formal identity
	\begin{equation}\label{eq:formalprod} \prod_{d \mid n} (1 + T^{\phi(d)}) = \sum_{m=0}^{n} T^m = \frac{1-T^{n+1}}{1-T}. \end{equation}
Evaluating \eqref{eq:formalprod} at $T=1$, we find that $2^{D} = n+1$, so that $n = 2^D-1$. In particular, $n$ is odd. Feeding the equality $n=2^D-1$ back into \eqref{eq:formalprod}, we find that
\begin{align*} \prod_{d \mid N}(1+T^{\phi(D)}) &= \frac{1-T^{2^D}}{1-T}\\
	&= (1+T)(1+T^2)(1+T^4)\cdots(1+T^{2^{D-1}}).\end{align*}
In both sides of this identity, we have a product of $D$ nonconstant  polynomials. Moreover, each of the $D$ right-hand factors is irreducible over $\Q$ (in fact, $1+T^{2^{j-1}} = \Phi_{2^j}(T)$). It follows from uniqueness of factorization in $\Q[T]$ that if one arranges the list of terms $\phi(d)$, where $d \mid n$, in increasing order, one obtains the sequence $\langle 1, 2, 4, \dots, 2^{D-1}\rangle$. 

The number $n$ must be squarefree. Otherwise, $p^2 \mid n$ for some $p \geq 3$ and so $\phi(p^2)=p(p-1)$ is divisible by $p$, contradicting that $\phi(p^2)$ is a power of $2$. So we may write
\[ n = p_1 \cdots p_k, \quad \text{where}\quad p_1 < p_2 < \dots < p_k. \]
Since each $\phi(d)$ is a power of $2$, we have in particular that each $p_i-1 = \phi(p_i)$ is a power of $2$, and so $p_i$ is a Fermat prime. Hence, the prime factorization of $n$ can be rewritten in the form
\[ n = F_{i_1} \cdots F_{i_k}, \quad \text{where}\quad 0 \leq i_1 < i_2 < \dots < i_k. \]
To complete the proof, we have to show that the sequence $\langle i_1, i_2, \dots, i_k\rangle$ coincides with the sequence $\langle 0, 1, \dots, k-1\rangle$. 

We claim that $T^n-1$ has a divisor of degree $m:= F_{i_1} F_{i_2} \cdots F_{i_{k-1}}+1$. To see this, it is sufficient (since $n$ is $\Q$-practical) to show that $m \leq n$. Clearly,
\begin{equation} m +1 \leq F_{0} F_1 F_2 \cdots F_{i_{k-1}} + 2\label{eq:throughout} = F_{i_{k-1}+1}.\end{equation}
Since $i_k \geq i_{k-1}+1$, we have $m +1 \leq F_{i_k} \leq n$. So $m < n$. Write $m$ as a sum of distinct terms $\phi(d)$, where $d\mid n$. If every $d$ involved in this representation divides $F_{i_1} \cdots F_{i_{k-1}}= n/F_{i_k}$, then $m \leq \sum_{d \mid F_{i_1} \cdots F_{i_{k-1}}} \phi(d) = F_{i_1} \cdots F_{i_{k-1}}$, 
which is not the case. So some $d$ in the representation is divisible by $F_{i_k}$, and hence $F_{i_k}-1 \leq \phi(d) \leq m$. Hence,
\[ F_{i_k} \leq m+1. \]
But from \eqref{eq:throughout}, we also have
\[ m+1 \leq F_{i_{k-1}+1} \leq F_{i_k}. \]
It follows that $i_k = i_{k-1}+1$ and that equality holds throughout \eqref{eq:throughout}. The latter forces the $(k-1)$-tuple $\langle i_1, i_2, \dots, i_{k-1}\rangle$ to coincide with the $(i_{k-1}+1)$-tuple $\langle 0, 1, 2, \dots, i_{k-1}\rangle$, so that $\langle i_1, \dots, i_{k-1} \rangle = \langle 0, 1, \dots, k-2\rangle$. Since $i_k = i_{k-1}+1$, we conclude that
$\langle i_1, i_2, \dots, i_k\rangle = \langle 0, 1, \dots, k-1\rangle$, as was to be shown.
\end{proof}

\begin{cor} There are precisely six $\Q$-optimal numbers, namely $2^{2^i}-1$ for $i=0, 1, \dots, 5$.
\end{cor}
\begin{proof} Since $F_0, F_1, \dots, F_4$ are prime while $F_5 = 641 \cdot 6700417$, Proposition \ref{prop:qefficient} shows that the $\Q$-optimal numbers are precisely the numbers $F_0 F_1 \cdots F_{i-1} = 2^{2^i}-1$ for $i = 0, 1, \dots, 5$.
\end{proof}  
%
%

\section*{Acknowledgements}
We thank Greg Martin and Carl Pomerance for helpful conversations. Some of the work on this paper was conducted while the first author was visiting Dartmouth College. He thanks the Dartmouth mathematics department for their hospitality.

\providecommand{\bysame}{\leavevmode\hbox to3em{\hrulefill}\thinspace}
\providecommand{\MR}{\relax\ifhmode\unskip\space\fi MR }
\providecommand{\MRhref}[2]{%
  \href{http://www.ams.org/mathscinet-getitem?mr=#1}{#2}
}
\providecommand{\href}[2]{#2}

\end{document}